\documentclass{article}

\usepackage{amsmath}
\usepackage{amsfonts}
\usepackage{amssymb}
\usepackage{amsthm}
\usepackage{enumerate}

\newtheorem{theorem}{Theorem}[section]
\newtheorem{lemma}[theorem]{Lemma}

\theoremstyle{definition}
\newtheorem{definition}[theorem]{Definition}

\newcommand{\ff}{\mathbb{F}_{2^n}}
\newcommand{\acc}{\mbox{AC}^0(\oplus)}
\newcommand{\bin}{\mathbb{Z}_{2^n-1}}
\newcommand{\summ}{\displaystyle\sum}
\newcommand{\F}{\mathbb{F}}
\newcommand{\E}{\mathbb{E}}
\newcommand{\pr}{\mathbb{P}}

\begin{document}
\title{Integer Addition and Hamming Weight}
\author{John Y. Kim}
\maketitle

\begin{abstract}
We study the effect of addition on the Hamming weight of a positive integer.  Consider the 
first $2^n$ positive integers, and fix an $\alpha$ among them.  We show that if the binary 
representation of $\alpha$ consists of $\Theta(n)$ blocks of zeros and ones, 
then addition by $\alpha$ causes a constant fraction of low 
Hamming weight integers to become high Hamming weight integers.  This result has applications 
in complexity theory to the hardness of computing powering maps using bounded-depth arithmetic circuits over 
$\F_2$.  Our result implies that powering by $\alpha$ composed of many blocks require 
exponential-size, bounded-depth arithmetic circuits over $\F_2$.  
\end{abstract}

\section{Introduction}
We begin with a natural, but largely unstudied question: How does the Hamming weight of an 
integer (written in base $2$) change under addition?  To make this precise, we take 
$\alpha\leq 2^n$ to be a fixed integer and let $S$ be chosen uniformly at random from 
$\{1,2,\cdots,2^n\}$.  Write $S$ in binary, and take $X$ to be its Hamming weight.  
Let $Y$ be the Hamming weight of the translation $S+\alpha$.  
Then what can we say about the joint distribution of initial and final weights, $(X,Y)$?

Our question is motivated by the problem of determining the complexity of powering maps 
in $\mathbb{F}_{2^n}$.  This problem has been studied extensively in complexity 
theory~\cite{Swastik,Eberly,FichTompa,Gathen,HealyViola,Hesse,Shparlinski}.  
Recently, Kopparty~\cite{Swastik} showed that the powering map 
$x\rightarrow x^{\frac{1}{3}}$ from $\ff\rightarrow\ff$ cannot be computed with a 
polynomial-size, bounded-depth arithmetic circuit over $\F_2$ (a.k.a $\acc$ circuit).  Recall that 
arithmetic circuits are only allowed addition and multiplication gates of unbounded fan-in).   
A major advantage of working in $\acc$ is that it is basis invariant.  
That is, determining the $\acc$ complexity of powering does not 
depend on the choice of basis for $F_{2^n}$.  
At the core of Kopparty's argument was the following shifting property of $\frac{1}{3}$: 
a constant fraction of elements in $\bin$ change from low to high Hamming 
weight under translation by $\frac{1}{3}$.

\begin{definition}
Let $M = \{x \in \bin \mid wt(x) \leq \frac{n}{2}\},$ where 
$wt(x)$ is the Hamming weight of $x$.  
We say that $\alpha \in \bin$ has the $\epsilon$-\emph{shifting property} if 
$M \cup (\alpha+M) \geq \left(\frac{1}{2}+\epsilon\right) 2^n$.
\end{definition}

We say that any binary string in $M$ is light, and any binary string not in $M$ is heavy.  
Then $\alpha$ has the $\epsilon$-shifting property if translating $\bin$ by $\alpha$ takes a constant 
fraction of light strings to heavy strings.  
Kopparty proved that powering by any $\alpha$ with the $\epsilon$-shifting property 
requires exponential circuit size in $\acc$~\cite{Swastik}.  Our main result 
is that any $\alpha$ with many blocks of $0$'s and $1$'s in its binary representation 
has the $\epsilon$-shifting property, proving a conjecture of Kopparty.

\begin{theorem}
\label{versatilePowers}
$\forall c>0$, $\exists \epsilon>0$, such that the following holds: 
Let $\sigma \in \{0,1\}^n$ be a bit-string of the form $\sigma = \sigma_1 \sigma_2 \cdots \sigma_m$, where $m \geq cn$, $\sigma_i$ is 
either $0^{L_i}$ or $1^{L_i}$, and each $L_i$ is chosen to be maximal.  Let $\alpha \in \mathbb{Z}_{2^n-1}$ have base $2$ representation 
given by $\sigma$.  Then $\alpha$ has the $\epsilon$-shifting property.
\end{theorem}

Note that the theorem still applies even in the setting of integer addition, not just 
when doing addition mod $2^n-1$.  
Our result states that $\alpha$ with $\Theta(n)$ blocks have the $\epsilon$-shifting 
property.  It is not difficult to show that $\alpha$ with $o(\sqrt{n})$ blocks do not have the 
$\epsilon$-shifting property.  First, observe that $o(\sqrt{n})$-sparse (i.e. $\alpha$ with 
Hamming weight $\leq o(\sqrt{n})$) 
$\alpha$ do not have the $\epsilon$-shifting property because addition by $\alpha$ can only 
increase the weight by $o(\sqrt{n})$.  Since there are $O(\frac{2^n}{\sqrt{n}})$ 
light binary strings of a fixed weight, we get $o(2^n)$ light strings changing 
to heavy strings under translation by $\alpha$.

Next, observe that any $\alpha$ with $o(\sqrt{n})$ blocks can be written as a 
difference of two $o(\sqrt{n})$-sparse strings: $\alpha = \beta - \gamma$.  
Since translating by $\alpha$ is equivalent to first translating by $\beta$ and 
then by $-\gamma$, we find that $\alpha$ with $o(\sqrt{n})$ blocks does not have 
the $\epsilon$-shifting property.  Thus, at least qualitatively, we see a strong 
connection between the $\epsilon$-shifting property and the number of blocks.  
Establishing a full characterization of the $\epsilon$-shifing property remains 
an interesting open question.

\subsection{Related Work}

Kopparty gave a different condition for when $\alpha$ has the $\epsilon$-shifting property: 
its binary representation consists mostly of a repeating constant-length string that is 
not all zeros or ones~\cite{Swastik}.  Note that any integer 
expressible as $\frac{a\cdot 2^n+b}{q}$, where $a,b,q\in\mathbb{Z}$, $q>1$ is odd, 
and $0<|a|,|b|<q$, has binary representation of this form.  As a consequence, taking 
$q$-th roots and computing $q$-th residue symbols cannot be done with 
polynomial-size $\acc$ circuits.  Our main result generalizes Kopparty's condition, 
as the periodic strings form a small subset of the strings with $\Theta(n)$ blocks.

Beck and Li showed that the $q$-th residue map is hard to compute in $\acc$ 
by using the concept of algebraic immunity~\cite{BeckLi}.  It is worth noting that their 
method does not say anything about the complexity of the $q$-th root map in $\acc$.  
So in this regard, there is something to be gained by analyzing the $\epsilon$-shifting property 
condition.  A more detailed history of the complexity of arithmetic operations using 
low-depth circuits can be found in \cite{Swastik}.

\section{Application}

It is known that powering by sparse $\alpha$ has polynomial-size circuits in $\acc$.  
Kopparty's work shows that powering by $\alpha$ with the $\epsilon$-shifting property require  
exponential-size circuits in $\acc$.  We will use this result, along with our new 
generalized criterion for when $\alpha$ has the $\epsilon$-shifting property to expand the 
class of $\alpha$ whose powers are difficult to compute in $\acc$.

The proof resembles the method of Razborov and Smolensky for showing that Majority is not in 
$\mbox{AC}^0(\oplus)$~\cite{Razborov,Smolensky,Smolensky2}.  
We can show for $\alpha$ with the $\epsilon$-shifting property that if powering by $\alpha$ is 
computable by an $\mbox{AC}^0(\oplus)$ circuit, 
then every function $f:F_{2^n}\rightarrow F_{2^n}$ is well-approximated by the sum of a low-degree polynomial with 
a function that sits in a low dimensional space.  The fact that there are not enough such functions 
provides the desired contradiction.  In this way, we show certain powers require exponential-size 
circuits in $\mbox{AC}^0(\oplus)$.

As a consequence of Theorem~\ref{versatilePowers} and the above Razborov-Smolensky method, 
we get that the powering by any $\alpha$ with $\Theta(n)$ maximal uniform blocks requires an
exponential-size $\acc$ circuit, thus greatly expanding the class of powers that are hard 
to compute in $\acc$.

\begin{theorem}
\label{largeComplexity}
Let $\alpha \in \mathbb{Z}_{2^n-1}$ have base $2$ representation in the form given by Theorem~\ref{versatilePowers}.

Define $\Lambda:\ff\rightarrow\ff$ by $\Lambda(x) = x^{\alpha}$.  

Then for every $\acc$ circuit 
$C:\ff\rightarrow\ff$ of depth $d$ and size $M\leq 2^{n^{\frac{1}{5}d}}$, for sufficiently 
large $n$ we have:

$$\text{Pr}[C(x) = \Lambda(x)] \leq 1-\epsilon_0,$$

where $\epsilon_0 > 0$ depends only on $c$ and $d$.
\end{theorem}

\section{The Proof of the Main Result}

\subsection{Outline of Proof}
Suppose we have a bit-string of length $n$.  The bit-string is called $\emph{light}$ if its Hamming weight is at most $\frac{n}{2}$.  
The bit-string is called $\emph{heavy}$ otherwise.  
It is enough to show that translation by $\alpha$ in $\mathbb{Z}_{2^n-1}$ transforms some 
positive constant fraction of the light bit-strings into heavy bit-strings.

We choose a binary string of length $n$ uniformly at random, translate it by $\alpha$, and look at the joint distribution of its 
initial weight $X$ and final weight $Y$.  Let $(\overline{X},\overline{Y}) = (X-\E[X],Y-\E[Y])$, so that when plotted, the plane is 
split into four quadrants.  The fraction of strings that shift weight from light 
to heavy is the proportion of the distribution in the second quadrant.  By symmetry, the same proportion of the distribution 
should lie in the fourth quadrant.  We will prove that some constant fraction of the distribution lies in the second or fourth quadrant.  

To get a handle on the distribution, we break up $\alpha$ 
into its $m$ uniform blocks of $0$'s and $1$'s, and consider addition on each block separately.  The distribution of the initial 
weight and final weight of any block is determined by the carry bit from the addition on the previous block and the carry bit 
going into the next block.  Thus, if the carry bits are given, then the weight distributions on the blocks are now independent.  
Although we will not be able to specify the distribution of the carry bits, we will show that with probability $\frac{1}{6}$, 
the carry bits have a certain property, and whenever they have this property, then the conditional distribution of $(X,Y)$ 
has a positive constant fraction of its mass in the second or fourth quadrants.

\subsection{Notation and Overview}

First, observe that it suffices to prove the main result for $M$ as viewed as a subset of $\mathbb{Z}_{2^n}$ instead of $\bin$.  
Note that only one element, $1^n \in \mathbb{Z}_{2^n}$, is not an element of $\mathbb{Z}_{2^n-1}$.  Also, when translating 
by $\alpha$, the resulting bit-string in $\mathbb{Z}_{2^n-1}$ is either the same or one more than the resulting bit-string 
in $\mathbb{Z}_{2^n}$.  Since only $o(n)$ of the heavy bit-strings of $\mathbb{Z}_{2^n}$ tranform into light bit-strings under 
translation by 1, if $\Theta(n)$ light bit-strings become heavy under translation by $\alpha$ in $\mathbb{Z}_{2^n}$, then at least 
$\Theta(n) - o(n) = \Theta(n)$ light bit-strings become heavy under translation by $\alpha$ in $\mathbb{Z}_{2^n-1}$.  This shows that 
we can work in the symmetric environment of all bit-strings of length $n$, $\mathbb{Z}_{2^n}$, and still achieve the result we want.

Let $S \in \mathbb{Z}_{2^n}$ be chosen uniformly at random.  Let $T = \alpha + S$.  Let $X = wt(S)$ and $Y = wt(T)$.  

Write $\alpha = \alpha_1 \alpha_2 \cdots \alpha_m$, $S = S_1 S_2 \cdots S_m$, and $T = T_1 T_2 \cdots T_m$, where each of the $i$-th parts 
have length $L_i$.  Let $X_i = wt(S_i)$ and $Y_i = wt(T_i)$.  Then $(X,Y) = \left(\displaystyle\sum_{i=1}^{m}{X_i},\displaystyle\sum_{i=1}^{m}{Y_i}\right)$.
Let $(\overline{X},\overline{Y}) = (X-\mathbb{E}[X],Y-\mathbb{E}[Y])$.  Then the part of the distribution of $(\overline{X},\overline{Y})$ in 
the second quadrant corresponds to light bit-strings translating to heavy bit-strings.  Similarly, the fourth quadrant corresponds to 
heavy to light bit-string translation.   To avoid having to pass to analogously defined $(\overline{X_i},\overline{Y_i})$ all the time, 
any reference to the second or fourth quadrant will be understood to be relative to $\left(\frac{L_i}{2},\frac{L_i}{2}\right)$ the mean of $(X_i,Y_i)$.  
We want to show that a positive constant fraction of the distribution lies in the second or fourth quadrants.

The random variables in the sum $\left(\displaystyle\sum_{i=1}^{m}{X_i},\displaystyle\sum_{i=1}^{m}{Y_i}\right)$ are highly dependent.  
To get around this, we will condition on the fixing of the carry bits.  Once the carry bits are fixed, the terms in the sum are independent.  
We will show that with probability at least $\frac{1}{6}$, we can find $\Theta(n)$ terms with identical distribution.  Since the terms are 
independent, we will use the multidimensional Central Limit Theorem to prove these identical distributions sum to a Gaussian distribution 
with dimensions of size $\Theta(\sqrt{n})$.

The remaining $O(n)$ terms can be divided into two categories.  Either the term has non-zero covariance matrix or it is 
a translation along the line $y=-x$ relative to the mean, $\left(\frac{L_i}{2},\frac{L_i}{2}\right)$.  
By applying the $2$-dimensional Chebyshev Inequality to the 
terms with non-zero covariance matrix, we show that at least half of the distribution lies in a square with dimensions $O(\sqrt{n})$.  
Any Gaussian with dimensions $\Theta(\sqrt{n})$ centered in the square of dimensions 
$O(\sqrt{n})$ will have a fixed positive proportion $p$ of its distribution in the second quadrant and $p$ of its distribution in the fourth quadrant.  
Finally, a translation of any magnitude along the line $y=-x$ still gives at least $p$ of the distribution in the second or fourth quadrant 
(although we don't know which one!).  However, 
as the addition map is a bijection from $\mathbb{Z}_{2^n}$ to itself, we get that the number of strings that go from light to heavy equals 
the nubmer of strings that go from heavy to light.  So we conclude that at least $p$ of the distribution lies in the second quadrant and 
at least $p$ of the distribution lies in the fourth quadrant.

\subsection{Computing the Distribution}

We first compute the $2$-dimensional distribution of the initial and final weights of the $i$-th block conditioned on the carry bit from the 
$(i+1)$-th block.  If the carry bit from the $i$-th block is denoted by $c_i$, then we want to understand the distribution of $(X_i,Y_i)$ 
given the carry bit $c_{i+1}$.  Suppose that $\alpha_i = 1^{L_i}$.  The case where $\alpha_i = 0^{L_i}$ is similar.  

\begin{lemma}
\label{dist}
Suppose that $\alpha_i = 1^{L_i}$.  The joint distribution of $(X_i,Y_i)$ conditioned on the carry bit $c_{i+1}$ is given by:

$$p_i(x,y\mid c_{i+1}=1) = 
\begin{cases} 
\frac{1}{2^{L_i}} \binom{L_i}{x} & \mbox{if }x=y (\mbox{then }c_i = 1) \\ 
0 &\mbox{else} 
\end{cases}$$

$$p_i(x,y\mid c_{i+1}=0) = 
\begin{cases}
\frac{1}{2^{L_i}} & \mbox{if }(x,y)=(0,L_i) (\mbox{then }c_i = 0)\\
\frac{1}{2^{L_i}} \binom{L_i-y+x-2}{x-1} & \mbox{if } L_i-1\geq y\geq x-1 \geq 0 (\mbox{then }c_i = 1)\\
\end{cases}$$
\end{lemma}

If $c_{i+1} = 1$, then $X_i$ = $Y_i$ and $c_i = 1$.  Hence, the probability mass function for $(X_i,Y_i)$ given $c_{i+1} = 1$ is given by

$$p_i(x,y\mid c_{i+1}=1) = 
\begin{cases} 
\frac{1}{2^{L_i}} \binom{L_i}{x} & \mbox{if }x=y \\ 
0 &\mbox{else} 
\end{cases}$$

If $c_{i+1} = 0$, then the distribution of $(X_i,Y_i)$ depends solely on the number of trailing zeros, $Z_i$, of $S_i$:

$$Y_i = \begin{cases} X_i &\mbox{if }X_{i+1}-Y_{i+1} = L_{i+1} \\ 
X_i+Z_i-1 &\mbox{if }Z_i<L_i \\
L_i &\mbox{if }Z_i=L_i\end{cases}$$

We therefore first compute the distribution of $Z_i$ conditioned on $X_i$ and use that to compute the joint distribution of $(X_i,Y_i)$.  
The distribution of $Z_i\mid X_i$ is given by

$$p_{Z_i}(z\mid x) = 
\begin{cases}
1 &\mbox{if } (x,z) = (0,L_i) \\
\frac{\binom{L_i-z-1}{x-1}}{\binom{L_i}{x}} &\mbox{if } L_i-z\geq x \\
0 &\mbox{else}
\end{cases}$$

Since $p_i(x,y\mid c_{i+1}) = p_{X_i}(x) p_{Y_i}(y\mid x)$, we compute $p_{X_i}(x)$ and $p_{Y_i}(y\mid x)$.  
As $X_i$ is binomial on $L_i$ trials with success probability $\frac{1}{2}$, 

$$p_{X_i}(x) = \frac{1}{2^{L_i}} \binom{L_i}{x} \mbox{for } x = 0,1,\cdots ,L_i.$$

We can also write the distribution of $Y_i\mid X_i$ in terms of the distribution of $Z_i\mid X_i$:

$$p_{Y_i}(y_i\mid x_i) = 
\begin{cases}
p_{Z_i}(y_i-x_i+1\mid x_i) &\mbox{if } 0\leq y_i-x_i+1 < L_i \\
p_{Z_i}(y_i-x_i\mid x_i) &\mbox{if } (x_i,y_i) = (0,L_i)
\end{cases}$$

Hence, we have the joint distribution of $(X_i,Y_i)$ is

$$p_i(x,y\mid c_{i+1}=0) = 
\begin{cases}
\frac{1}{2^{L_i}} & \mbox{if }(x,y)=(0,L_i) (\mbox{then }c_i = 0)\\
\frac{1}{2^{L_i}} \binom{L_i-y+x-2}{x-1} & \mbox{if } L_i-1\geq y\geq x-1 \geq 0 (\mbox{then }c_i = 1)\\
\end{cases}$$

Similarly, if $\alpha_i = 0^{L_i}$, then the distribution of $(X_i,Y_i)$ is as follows:

If $c_{i+1} = 0$, then $X_i=Y_i$, $c_{i} = 0$ and 

$$p_i(x,y\mid c_{i+1}=0) = 
\begin{cases} 
\frac{1}{2^{L_i}} \binom{L_i}{x} & \mbox{if }x=y (\mbox{then }c_i = 0) \\ 
0 &\mbox{else} 
\end{cases}$$

When the carry bit makes the addition trivial, we call the resulting distribution the trivial distribution.  
Otherwise, the carry bit $c_{i+1} = 1$.  In this case, the distribution of $(X_i,Y_i)$ turns 
out to be symmetric with the case where $\alpha_i = 1^{L_i}$ and $c_{i+1} = 0$:

$$p_i(x,y\mid c_{i+1}=1) = 
\begin{cases}
\frac{1}{2^{L_i}} & \mbox{if }(x,y)=(L_i,0) (\mbox{then }c_i = 1)\\
\frac{1}{2^{L_i}} \binom{L_i-x+y-2}{y-1} & \mbox{if } L_i-1\geq x\geq y-1 \geq 0 (\mbox{then }c_i = 0)\\
\end{cases}$$

When the carry bit makes the addition nontrivial, as in this case, we call the resulting distribution the 
nontrivial distribution.

Fixing the carry bits leads to four types of distributions for the blocks based on the carry bit coming in 
from the previous block addition and the resulting carry bit from the current block addition.

\begin{enumerate}
 \item The block distribution is trivial and produces a carry bit that makes the subsequent block 
 distribution non-trivial (Trivial to non-trivial).
 \item Non-trivial to trivial
 \item Non-trivial to non-trivial (block length $L=1$)
 \item Non-trivial to non-trivial (block length $L\geq2$)
\end{enumerate}

We make the distinction between block lengths $1$ and $2$ for non-trivial to non-trivial distributions as 
the latter is the only distribution with invertible covariance matrix.  Ideally, we will find many 
identical distributions of type $4$, which will sum to a Gaussian with large enough dimensions.  
This will not be possible when most of the blocks have length 1, which we deal with separately.

Knowing the weight distribution of a block given the previous carry, it is straightforward to write down 
the distributions given both the previous carry and the produced carry.  Again, we assume the block 
$\alpha_i = 1^{L_i}$.

As a trivial distribution always produces a non-trivial carry, we get the trivial to non-trivial distribution 
is the same is the trivial distribution:

$$p_i(x,y\mid c_{i+1}=1, c_i=1) = 
\begin{cases} 
\frac{1}{2^{L_i}} \binom{L_i}{x} & \mbox{if }x=y  \\ 
0 &\mbox{else} 
\end{cases}$$

A non-trivial distribution that produces a trivial carry must have $(X_i,Y_i) = (0,L_i)$.  Also, 
a non-trivial distribution of block length $1$ that produces a non-trivial carry must have $(X_i,Y_i) = (1,0)$.  

Finally, a non-trivial distribution of block length greater than $1$ that produces a non-trivial carry 
has distribution:

$$p_i(x,y\mid c_{i+1}=0, c_i=1) = 
\begin{cases}
\frac{1}{2^{L_i}-1} \binom{L_i-y+x-2}{x-1} & \mbox{if } L_i-1\geq y\geq x-1 \geq 0 \\
0 &\mbox{else}
\end{cases}$$

We summarize these distributions in the next lemma:

\begin{lemma}
\label{dist2}
Suppose that $\alpha_i = 1^{L_i}$.  The joint distribution of $(X_i,Y_i)$ conditioned on the carry bits 
$c_{i+1}$ and $c_i$ is given by:

$$p_i(x,y\mid c_{i+1}=1, c_i=1) = 
\begin{cases} 
\frac{1}{2^{L_i}} \binom{L_i}{x} & \mbox{if }x=y  \\ 
0 &\mbox{else} 
\end{cases}$$

$$p_i(x,y\mid c_{i+1}=0, c_i=0) = 
\begin{cases}
1 & \mbox{if } (x,y) = (0,L_i) \\
0 &\mbox{else}
\end{cases}$$

$$p_i(x,y\mid c_{i+1}=0, c_i=1) = 
\begin{cases}
\frac{1}{2^{L_i}-1} \binom{L_i-y+x-2}{x-1} & \mbox{if } L_i-1\geq y\geq x-1 \geq 0 \\
0 &\mbox{else}
\end{cases}$$
\end{lemma}

Observe that the last non-trivial to non-trivial probability distribution works for all lengths $L_i\geq 1$.  However, when 
$L_i=1$, $(x,y) = (0,1)$ with probability $1$.  We will still consider this as a separate type of distribution as its 
covariance matrix is all zeros, and consequently not invertible, which will be important for analysis.

\subsection{Computing the Covariance Matrix}

\begin{lemma}
\label{covariance}
The covariance matrix $M$ of the trivial to non-trivial distribution of the random vector $(X_i,Y_i)$ is given by 
$$M = 
\begin{pmatrix}
\frac{L_i}{4} & \frac{L_i}{4} \\
\frac{L_i}{4} & \frac{L_i}{4}
\end{pmatrix}$$

The covariance matrix $M$ of the non-trivial  to non-trivial distribution of the random vector $(X_i,Y_i)$ is given by 
$$M = 
\begin{pmatrix}
c & d \\
d & c
\end{pmatrix}$$

where $c = \frac{L_i}{4}\left(1+\frac{1}{2^{L_i}-1}\right) - \frac{L_i^2}{4}\left(1+\frac{1}{2^{L_i}-1}\right)\frac{1}{2^{L_i}-1},$

and $d = \frac{L_i}{4}\left(1+\frac{1}{2^{L_i}-1}\right) + \frac{L_i^2}{4}\left(1+\frac{1}{2^{L_i}-1}\right)\frac{1}{2^{L_i}-1} -1.$

\end{lemma}

\begin{proof}

To simplify our notation, let $(X(L),Y(L))$ denote some $(X_i,Y_i)$ with $L_i = L$.  
We begin with the trivial to non-trivial distribution.  
Since $X(L)$ is binomial on $L$ trials with success probability $\frac{1}{2}$, 
$Var(X(L)) = \frac{L}{4}$.  Since the bit string corresponding to $Y(L)$ can be viewed 
as a translation of the bit string corresponding to $X(L)$ in $\mathbb{Z}_{2^L}$, the 
distribution of $Y(L)$ is the same as the distribution of $X(L)$.  Hence, 
$Var(Y(L)) = \frac{L}{4}$.  It remains to compute $Cov(X(L),Y(L))$.  
In the case of the trivial distribution, $X(L) = Y(L)$.  So $Cov(X(L),Y(L)) = Var(X(L)) = \frac{L}{4}$.

The case of the non-trivial to non-trivial distribution requires more work.  
For our computation, we assume $\alpha_i = 1^{L_i}$.  As the nontrivial distributions are symmetric in $x$ and 
$y$, the covariances will be the same.  We begin by evaluating $Var(X(L)) = \E[X(L)^2]-\E[X(L)]^2.$  

As the block with weight $X(L)$ is chosen uniformly at random among all non-zero strings of length $L$, 

\begin{eqnarray*}
\E[X(L)] &=& \frac{L}{2}\left(1+\frac{1}{2^L-1}\right) \\
\E[X(L)^2] &=& \frac{1}{2^L-1} \displaystyle\sum_{n=1}^L {\binom{L}{n}n^2}.
\end{eqnarray*}

We use repeated differentiation of the binomial theorem to compute $\displaystyle\sum_{n=1}^L {\binom{L}{n}n^2}$.

$$\displaystyle\sum_{n=1}^L {\binom{L}{n}x^n} = (x+1)^L-1.$$

Differentiating with respect to $x$ yields:

\begin{eqnarray*}
\displaystyle\sum_{n=1}^L {\binom{L}{n}nx^{n-1}} &=& L(x+1)^{L-1} \\
\displaystyle\sum_{n=1}^L {\binom{L}{n}nx^{n}} &=& L(x+1)^{L-1}x.
\end{eqnarray*}

Differentiating a second time with respect to $x$ gives:

$$\displaystyle\sum_{n=1}^L {\binom{L}{n}n^2 x^{n-1}} = L(L-1)(x+1)^{L-2}x + L(x+1)^{L-1}.$$

Plugging in $x=1$ gives us the sum we want:

\begin{eqnarray*}
\displaystyle\sum_{n=1}^L {\binom{L}{n}n^2} &=& L(L-1)2^{L-2} + L2^{L-1} \\
&=& L2^L \left(\frac{L-1}{4} + \frac{1}{2}\right) \\
&=& \frac{L(L+1)}{4}2^L.
\end{eqnarray*}

Hence, the variance of $X(L)$ is given by:

\begin{eqnarray*}
Var(X(L)) &=& \frac{L(L+1)}{4}\frac{2^L}{2^L-1} - \frac{L^2}{4}\left(1+\frac{1}{2^L-1}\right)^2 \\
&=& \frac{L^2+L}{4}\left(1+\frac{1}{2^L-1}\right) - \frac{L^2}{4}\left(1+\frac{1}{2^L-1}\right) 
- \frac{L^2}{4}\left(1+\frac{1}{2^L-1}\right)\frac{1}{2^L-1} \\
&=& \frac{L}{4} \left(1+\frac{1}{2^L-1}\right) - \frac{L^2}{4} \left(1+\frac{1}{2^L-1}\right)\frac{1}{2^L-1}.
\end{eqnarray*}

Observe that $Y$ is the weight of a block of length $L$ chosen uniformly at random from all strings except $1^L$.  
So by symmetry, $Var(Y) = Var(X)$.  We now compute $Cov(X,Y) = \E[X(L),Y(L)] - \E[X(L)]\E[Y(L)]$.  

\begin{eqnarray*}
\mathbb{E}[X(L)Y(L)] &=& \frac{1}{2^L-1} \displaystyle\sum_{1\leq y\leq x+1\leq L}{xy\binom{L-x+y-2}{y-1}} \\
&=& \frac{1}{2^L-1} \displaystyle\sum_{x=0}^{L-1}{x\displaystyle\sum_{y=1}^{x+1}{y\binom{L-x+y-2}{y-1}}} \\
&=& \frac{1}{2^L-1} \displaystyle\sum_{x=0}^{L-1}{x\displaystyle\sum_{y=0}^{x}{(y+1)\binom{L-x+y-1}{y}}}.
\end{eqnarray*}

Let $A(x) = \displaystyle\sum_{y=0}^{x}{(y+1)\binom{L-x+y-1}{y}}$ be the inner summation.  Then by repeated application 
of the hockey stick identity, 

\begin{eqnarray*}
A(x) &=& \displaystyle\sum_{y=0}^{x}{(x+1)\binom{L-x+y-1}{y}} - \displaystyle\sum_{y=0}^{x-1}{(x-y)\binom{L-x+y-1}{y}} \\
&=& (x+1)\binom{L}{x} - \displaystyle\sum_{y=0}^{x-1}\displaystyle\sum_{j=0}^{y}{\binom{L-x+j-1}{j}} \\
&=& (x+1)\binom{L}{x} - \displaystyle\sum_{y=0}^{x-1}{\binom{L-x+y}{y}} \\
&=& (x+1)\binom{L}{x} - \binom{L}{x-1}.
\end{eqnarray*}

Substituting $A(x)$ back into our expression for $\mathbb{E}[X(L)Y(L)]$ yields

$$\mathbb{E}[X(L)Y(L)] = \frac{1}{2^L-1}\left[\displaystyle\sum_{x=1}^{L-1}{x(x+1)\binom{L}{x}} - \displaystyle\sum_{x=1}^{L-1}{x\binom{L}{x-1}} \right].$$

Let $B = \displaystyle\sum_{x=1}^{L-1}{x(x+1)\binom{L}{x}}$ and let $C = \displaystyle\sum_{x=1}^{L-1}{x\binom{L}{x-1}}$.  
We can simplify $B$ and $C$ by starting with the binomial theorem and applying standard generating function methods.

\begin{eqnarray*}
\displaystyle\sum_{k=0}^{L}{\binom{L}{k}x^k} &=& (1+x)^L \\
\displaystyle\sum_{k=0}^{L}{\binom{L}{k}x^{k+1}} &=& x(1+x)^L.
\end{eqnarray*}

Differentiating both sides with respect to $x$ gives:

\begin{equation}
\label{binomDiff1}
\displaystyle\sum_{k=0}^{L}{(k+1)\binom{L}{k}x^k} = (1+x)^{L-1}(1+(L+1)x).
\end{equation}

Substituting $x=1$ in equation~\ref{binomDiff1} gives

\begin{eqnarray*}
\displaystyle\sum_{k=0}^{L}{(k+1)\binom{L}{k}} &=& (1+x)^{L-1}(L+2) \\
C+L\binom{L}{L-1}+(L+1)\binom{L}{L} &=& 2^{L-1}(L+2) \\
C &=& 2^{L-1}(L+2) - (L^2+L+1).
\end{eqnarray*}

To get $B$, we differentiate equation~\ref{binomDiff1} with respect to $x$ once more

\begin{equation}
\label{binomDiff2}
\displaystyle\sum_{k=1}^{L}{k(k+1)\binom{L}{k}x^{k-1}} = L(1+x)^{L-2}(2+(L+1)x).
\end{equation}

Substituting $x=1$ in equation~\ref{binomDiff2} gives

\begin{eqnarray*}
\displaystyle\sum_{k=1}^{L}{k(k+1)\binom{L}{k}} &=& L 2^{L-2}(L+3) \\
B+L(L+1)\binom{L}{L} &=& 2^{L-2}L(L+3) \\
B &=& 2^{L-2}L(L+3) - (L^2+L).
\end{eqnarray*}

Using the simplified expressions for $B$ and $C$, we get

\begin{eqnarray*}
\mathbb{E}[X(L)Y(L)] &=& \frac{1}{2^L-1}(B-C)\\
&=& \frac{1}{2^L-1}(2^{L-2}L(L+3) - (L^2+L) - 2^{L-1}(L+2) + (L^2+L+1)) \\
&=& \frac{2^L}{2^L-1} \left(\frac{L(L+1)}{4} - \frac{L(L+1)}{2^L} - \frac{L+2}{2} + \frac{L^2+L+1}{2^L}\right) \\
&=& \left(1+\frac{1}{2^L-1}\right) \left(\frac{L^2+L-4}{4} + \frac{1}{2^L}\right) \\
&=& \left(1+\frac{1}{2^L-1}\right) \left(\frac{L^2+L}{4} - \left(1-\frac{1}{2^L}\right)\right) \\
&=& \frac{L^2+L}{4} \left(1+\frac{1}{2^L-1}\right) - 1 \\
\end{eqnarray*}

Hence the covariance of X(L) and Y(L) is

\begin{eqnarray*}
Cov(X(L),Y(L)) &=& \frac{L^2+L}{4} \left(1+\frac{1}{2^L-1}\right) - 1 - 
\frac{L}{2}\left(1+\frac{1}{2^L-1}\right) \frac{L}{2}\left(1-\frac{1}{2^L-1}\right) \\
&=& \frac{L^2+L}{4} \left(1+\frac{1}{2^L-1}\right) - 1 - 
\frac{L^2}{4}\left(1+\frac{1}{2^L-1}\right) + \frac{L^2}{4}\left(1+\frac{1}{2^L-1}\right) \frac{1}{2^L-1} \\
&=& \frac{L}{4} \left(1+\frac{1}{2^L-1}\right) + \frac{L^2}{4}\left(1+\frac{1}{2^L-1}\right) \frac{1}{2^L-1} - 1.
\end{eqnarray*}

\end{proof}

\subsection{Breaking up the Weight Distribution}

Recall that the joint distribution of initial and final weights is the sum of the 
joint distribution of initial and final weights for 
$m = \Theta(n)$ smaller parts: $(X,Y) = \displaystyle\sum_{i=1}^{m}{(X_i,Y_i)}$.  
However, the terms of this sum are dependent.  We can remove the dependence by 
first sampling the carry bits according to their distribution.  Given the carry bits, 
all of the terms in the sum are independent and have one of four types of distributions 
given by Lemma~\ref{dist2}.  This gives us access to the Central Limit Theorem and 
the fact that covariance matrices add, both of which will be used in the proof.

We will break up $(X,Y)$ into a sum of a Gaussian $(X_G,Y_G)$, a translation 
$(X_T,Y_T)$ and some remainder $(X_R,Y_R)$, which we show is well-behaved.  
As the non-trivial to non-trivial distribution with block length at least $2$ (Type 
$4$) is the only type with invertible covariance matrix, our goal will be to find many 
identical distributions of this type.  By the Central Limit Theorem, these sum to a $2$-D 
Gaussian $(X_G,Y_G)$ of dimensions $\Theta(\sqrt{n})$.  
This will be the main part of the sum that pushes the 
distribution into the second and fourth quadrants.

It is not always possible to find many identical distributions of type $4$.  If there are 
$o(n)$ blocks of length at least $2$, then it is trivially impossible.  We deal with this 
case separately with a slightly modified argument.  Otherwise, there are $\Theta(n)$ blocks 
of length at least $2$.  We will show that with probability at least $\frac{1}{6}$, the 
carry bits arrange themselves in such a way so that there are $\Theta(n)$ distributions of 
type $4$.  This is enough to find many identical distributions of type $4$.

We then consider the sum of the remainder of the type $4$ distributions along with the 
trivial to non-trivial type $1$ distributions, $(X_R,Y_R)$, 
and show that it is well-behaved.  As the covariance 
matrices add, we will be able to apply the $2$-D Chebyshev inequality to guarantee that 
half of the distribution lies inside an ellipse of dimensions $O(\sqrt{n})$.  
This will be enough to guarantee some constant proportion $p$ of the distribution of 
$(X_G,Y_G) + (X_R,Y_R)$ in the second quadrant, and the same proportion $p$ in the 
fourth quadrant.  The rest of the distribution of $(X,Y)$ is a translation $(X_T,Y_T)$ 
along the line $y=-x$ relative to the mean.  After translation, we still have $p$-fraction 
of the distribution in either the second or fourth quadrant.

We first consider the case where there are $m'=\Theta(n)$ blocks of length at least $2$.  
The following lemma says that with probability at least $\frac{1}{6}$, we get many 
distributions of type $4$.  

\begin{lemma}
\label{manyType4}
Suppose there are $m' = \Theta(n)$ blocks of length at least $2$.  Let $X$ be the number 
of non-trivial to non-trivial distributions with block length at least $2$.  Then:
$$\pr(X>\frac{m'}{4}) > \frac{1}{6}.$$
\end{lemma}

\begin{proof}

\begin{eqnarray*}
\E[X] &=& \summ_{i=1}^{m'}{\pr(\text{Block } i \text{ is non-trivial to non-trivial})} \\
&=& \summ_{i=1}^{m'}{\pr(\text{Non-trivial carry out } \mid \text{ non-trivial carry in})\cdot \pr(\text{Non-trivial carry in}} \\
&\geq& \summ_{i=1}^{m'}{\frac{3}{4}\cdot \frac{1}{2}} \\
&=& \frac{3}{8}m'.
\end{eqnarray*}

Let $Y$ be the number of blocks of length at least $2$ that are not non-trivial to non-trivial.  
Then $\E[Y] \leq \frac{5}{8}m'$.  By Markov's inequality, 

$$\pr(Y\geq t) \leq \frac{\E[Y]}{t} \leq \frac{5}{8}\cdot \frac{m'}{t}.$$

Taking $t = \frac{3}{4}m'$ yields 

\begin{eqnarray*}
\pr(Y\geq \frac{3}{4}m') &\leq& \frac{5}{6} \\
\pr(Y\leq \frac{3}{4}m') &\geq& \frac{1}{6} \\
\pr(X\geq \frac{1}{4}m') &\geq& \frac{1}{6}.
\end{eqnarray*}

\end{proof}

We now show that many distributions of type $4$ implies many identical distributions 
of type $4$.

\begin{lemma}
\label{identicalDists}
Suppose we have $m=\Theta(n)$ bit-strings of total length at most $n$.  
Then there is some fixed positive length $L$ such that $l=\Theta(n)$ 
bit-stings have length $L$.
\end{lemma}

\begin{proof}

Let $n_k$ be the number of blocks of length $k$.  Then we have the following equations 
about the total number of blocks and the total length of all the blocks:

\begin{eqnarray*}
\displaystyle\sum_{k=1}^{n}{n_k} &=& m\\
\displaystyle\sum_{k=1}^{n}{k\cdot n_k} &\leq& n.
\end{eqnarray*}

Dividing the above equations by $m$ yields:

\begin{eqnarray*}
\displaystyle\sum_{k=1}^{n}{\frac{n_k}{m}} &=& 1\\
\displaystyle\sum_{k=1}^{n}{k\cdot \frac{n_k}{m}} &\leq& \frac{n}{m}.
\end{eqnarray*}

Now consider the random variable $A$ that returns the length of a block chosen uniformly at random.  
The left hand side of the second equation is the expected value of $A$:

$$\mathbb{E}[A] \leq \frac{n}{m}.$$

Since $m=\Theta(n)$, $\frac{n}{m}$ is bounded below by a constant.  More precisely, there is a 
constant $c_1 > 0$ such that for large enough $n$, $m > c_1n$.  So:

$$\mathbb{E}[A] \leq \frac{n}{m} < \frac{1}{c_1}.$$

By Markov's Inequality combined with the upper bound on the expected length, we have:

$$\mathbb{P}\{A>\frac{k}{c_1}\} \leq \frac{\mathbb{E}[A]}{\frac{k}{c_1}} < \frac{1}{k}.$$

This means that at most $\frac{m}{k}$ blocks have length greater than $\frac{k}{c_1}$.  
So there must be at least 
$\left(1 - \frac{1}{k}\right)m$ non-trivial blocks of length at most $\frac{k}{c_1}$.  
By the Pidgeonhole Principle, there is some length which is at most $\frac{k}{c_1}$, 
that appears $\frac{\left(1 - \frac{1}{k}\right)m}{\frac{k}{c_1}} = 
\frac{1}{k}\left(1 - \frac{1}{k}\right)c_1 m$ times.  

Essentially, some non-trivial block of short length must appear very often.  
We should pick the value of $k$ that maximizes the frequency of this length: $k=2$.  
We get that some non-trivial block of short length appears at least $\frac{c_1}{4}m$ times.  Taking 
$c = \frac{c_1}{4}$, then we get that the number 
of non-trivial distributions $l \geq cm$, and this constant $c$ is independent of the 
assignment of carry bits.

\end{proof}

So with probability at least $\frac{1}{6}$, we can find $l=\Theta(n)$ identical type $4$ distributions.  
Since these identical distributions are independent of each other, the Central Limit Theorem together with 
Lemma~\ref{covariance} tells us that the distribution of their sum is a Gaussian with covariance matrix given by 

$$M_G = 
\begin{pmatrix}
cl & dl \\
dl & cl
\end{pmatrix}$$

where $c = \frac{L}{4}\left(1+\frac{1}{2^{L}-1}\right) - \frac{L^2}{4}\left(1+\frac{1}{2^{L}-1}\right)\frac{1}{2^{L}-1},$ 

$d = \frac{L}{4}\left(1+\frac{1}{2^{L}-1}\right) + \frac{L^2}{4}\left(1+\frac{1}{2^{L}-1}\right)\frac{1}{2^{L}-1} -1,$ 
and $L\geq 2$.  

\begin{lemma}
\label{gaussDist}
Suppose $G$ is a 2-dimensional Gaussian distribution with covariance matrix given by $M_G$.  Then a fixed positive proportion 
of the distribution of $G$ lies inside (and outside) an ellipse centered at the mean with dimensions $\Theta(\sqrt{n})$.    
Furthermore, the probability density function $f_G(x,y)\geq \frac{1}{\pi n} e^{-\frac{144n}{l}}$ inside a circle of radius 
$4\sqrt{n}$ centered at the mean of $G$.  
\end{lemma}

\begin{proof}
Observe first that $c-d = 1-\frac{L^2}{2}\left(1+\frac{1}{2^{L}-1}\right)\frac{1}{2^{L}-1} \geq \frac{1}{9}$ 
and $d \geq \frac{1}{9}$ for any value of $L\geq 2$.  As $\det(M_G) = (c^2-d^2)l^2 \neq 0$, $M_G$ is invertible.  
So letting $\overline{G} = (\overline{X_G},\overline{Y_G})$ denote the distribution obtained by translating $G$ 
to its mean, we get that a fixed proportion of the distribution lies in the ellipse defined by:

$$\begin{pmatrix}
\overline{X_G} & \overline{Y_G}
\end{pmatrix} M_G^{-1} 
\begin{pmatrix}
\overline{X_G}\\
\overline{Y_G}
\end{pmatrix} = 2$$

The inverse of $M_G$ is given by: 

$$M_G^{-1} = \frac{1}{(c^2-d^2)l^2} \begin{pmatrix}
cl & -dl \\
-dl & cl
\end{pmatrix}$$

Substituting $M_G^{-1}$ back into the equation of the ellipse gives:

\begin{eqnarray*}
\frac{1}{(c^2-d^2)l}[c\overline{X_G}^2 - 2d\overline{X_G}\overline{Y_G} + c\overline{Y_G}^2] &=& 2 \\
\overline{X_G}^2 - \frac{2d}{c}\overline{X_G}\overline{Y_G} + \overline{Y_G}^2 &=& \frac{2(c^2-d^2)}{c}l.
\end{eqnarray*}

We have an equation of the form $x^2-2axy+y^2 = b$, where $a = \frac{d}{c} < 1$.  
This describes an ellipse rotated by $\frac{\pi}{4}$ counterclockwise.  
By rotating the ellipse clockwise by $\frac{\pi}{4}$, we can find the dimensions of the ellipse.  Making the substitution: 

$$\begin{pmatrix}x \\ y\end{pmatrix} = 
\begin{pmatrix} \frac{\sqrt{2}}{2}(x'+y')\\ \frac{\sqrt{2}}{2}(y'-x')\end{pmatrix},$$

we get that the equation of the rotated ellipse is

\begin{equation}
\label{ellipse45}
\frac{x^2}{\frac{b}{1+a}} + \frac{y^2}{\frac{b}{1-a}} = 1.
\end{equation}

Taking $a = \frac{d}{c}$ and $b = \frac{2(c^2-d^2)}{c}l$ as in the ellipse for our Gaussian, we find that the 
squares of the dimensions of the ellipse are given by:

\begin{eqnarray*}
\frac{b}{1-a} &=& \frac{2(c^2-d^2)}{c}l \cdot \frac{c}{c-d} = 2(c+d)l \geq \frac{2}{3}l \\
\frac{b}{1+a} &=& \frac{2(c^2-d^2)}{c}l \cdot \frac{c}{c+d} = 2(c-d)l \geq \frac{2}{9}l.
\end{eqnarray*}

Hence, both dimensions of the ellipse are $\Theta(\sqrt{n})$.  In fact, both dimensions exceed $\frac{1}{3}\sqrt{l}$.  
So the circle of radius $\frac{1}{3}\sqrt{l}$ centered at the mean lies completely inside the ellipse.  Scaling every 
dimension up by a factor of $\sqrt{\frac{144n}{l}}$ tells us that inside the circle of radius $4\sqrt{n}$ 
centered at the mean, we have:

$$\begin{pmatrix}
\overline{X_G} & \overline{Y_G}
\end{pmatrix} M_G^{-1} 
\begin{pmatrix}
\overline{X_G}\\
\overline{Y_G}
\end{pmatrix} \leq \frac{144n}{l}.$$

Therefore, we have the following lower bound on the probability distribution function inside the circle 
of radius $4\sqrt{n}$:

\begin{eqnarray*}
f_G(x,y) &\geq& \frac{1}{2\pi\sqrt{\det{M_G}}} e^{-\frac{144n}{l}} \\
&=& \frac{1}{2\pi\sqrt{(c^2-d^2)l^2}} e^{-\frac{144n}{l}} \\
&\geq & \frac{1}{2\pi cl} e^{-\frac{144n}{l}} \\
&\geq & \frac{1}{\pi n} e^{-\frac{144n}{l}}.
\end{eqnarray*}

\end{proof}

The above sequence of lemmas can be used to show that if we start with many blocks of length at least $2$, 
then with positive constant probability, we can find many identical distributions that sum to a Gaussian 
of dimensions $\Theta(\sqrt{n})$.  Suppose now that the exponent $\alpha$ has a total of $m$ blocks, 
but fewer than $0.01m$ blocks of length at least $2$.  
Then at least $0.99$ fraction of the blocks have length $1$.  Consider all consecutive 
block pairs.  At most $0.01$ fraction of these pairs have their first block with length $2$, and at most 
$0.01$ fraction have their second block with length at least $2$.  So at most $0.02$ fraction have a 
block of length at least $2$.  Hence, $0.98$ fraction of the pairs consists of two blocks of length $1$.  
By the Pidgeonhole Principle, at least $0.49$ fraction of the pairs are either all $01$ or all $10$.  
Without loss of generality, assume that $0.49$ fraction of consecutive block pairs are $01$.  
We now treat each block pair $01$ as a single block of length $2$.  The initial and final weight 
distribution of this larger block, given there is no carry in and no carry out matches the type 
$4$ distribution.  We have proven the existence of a large number of modified blocks of length $2$:

\begin{lemma}
\label{manyBlocks}
Suppose there are fewer than $0.01$ fraction of the blocks have length at least $2$.  
Then at least $0.49$ fraction of consecutive block pairs are $01$ or at least 
$0.49$ fraction of consecutive block pairs are $10$.
\end{lemma}

Lemma~\ref{manyBlocks} essentially reduces the case of having few blocks of length at least $2$ to 
the case where there are many blocks of length at least $2$ by consolidating many of the length $1$ blocks.  
As there are $\Theta(n)$ such consolidated blocks, and each has type $4$ distribution with 
probability at least $\frac{3}{8}$, we can again find $\Theta(n)$ identical type $4$ distributions 
by Lemma~\ref{manyType4}.  Lemma~\ref{gaussDist} then says that these identical distributions sum 
to a Gaussian of large dimensions.  So for any $\alpha$, we can find many terms in the initial and 
final weight distribution summing to a large Gaussian.

\subsection{Distribution of the Sum of the Remaining Terms}

Consider the terms remaining in the distribution of $(X,Y) = \displaystyle\sum_{i=1}^{m}{(X_i,Y_i)}$ when the terms 
contributing to the Gaussian are removed.  The terms with distribution types $2$ or $3$ are translations in the 
$<1,-1>$ direction relative to $\left(\frac{L}{2},\frac{L}{2}\right)$.  These will contribute to the translation 
part of the distribution $(X_T,Y_T)$.  The rest of the terms of type $4$ along with the terms of type $1$ sum 
to the remainder $R=(X_R,Y_R)$.  There are $O(n)$ terms remaining.  By Lemma~\ref{covariance}, the covariance 
matrix of each of these terms is one of the following two forms:

$$\begin{pmatrix}
\frac{L}{4} & \frac{L}{4} \\
\frac{L}{4} & \frac{L}{4}
\end{pmatrix}$$

$$\begin{pmatrix}
c & d \\
d & c
\end{pmatrix}$$

where $c = \frac{L}{4}\left(1+\frac{1}{2^{L}-1}\right) - \frac{L^2}{4}\left(1+\frac{1}{2^{L}-1}\right)\frac{1}{2^{L}-1},$ 
and $d = \frac{L}{4}\left(1+\frac{1}{2^{L}-1}\right) + \frac{L^2}{4}\left(1+\frac{1}{2^{L}-1}\right)\frac{1}{2^{L}-1} -1.$

As the terms are independent given a fixing of the carry bits, the covariance matrices add.  The total covariance matrix 
of the sum is:

$$M_R = 
\begin{pmatrix}
C & D \\
D & C
\end{pmatrix}$$

where $D \leq C \leq \frac{n}{3}$, with $D=C$ only when the remainder is a sum of type $1$ distributions.

\begin{lemma}
\label{remainder}
At least half of the distribution of the remainder lies in a circle of radius $\sqrt{2n}$ 
centered at the mean.
\end{lemma}

\begin{proof}
When the remainder is a sum of type $1$ distributions, then the remainder has the form $(X_R,X_R)$, 
where $X_R$ is a biniomial distribution with $L_R\leq n$ trials and success probability $\frac{1}{2}$.  
So by Chebyshev's Inequality, 

\begin{eqnarray*}
\pr(|X_R - \frac{L_R}{2}| > \sqrt{\frac{L_R}{2}}) \leq \frac{1}{2} \\
\pr(|X_R - \frac{L_R}{2}| > \sqrt{\frac{n}{2}}) \leq \frac{1}{2} \\
\pr(|X_R - \frac{L_R}{2}| \leq \sqrt{\frac{n}{2}}) \geq \frac{1}{2}.
\end{eqnarray*}

So in this case, at least half of the distribution of the remainder lies in a circle of radius $\sqrt{\frac{n}{2}}$.
When the remainder contains some type $4$ distributions, then $D < C \leq \frac{n}{3}$.  Hence, the 
covariance matrix $M_R$ is invertible.  So we may apply the $2$-dimensional Chebyshev inequality to 
$(X_R,Y_R)$ to get:

$$\pr\{\overline{R}M_R^{-1}\overline{R}^T > t\} \leq  \frac{2}{t^2}.$$

Taking $t=2$ yields

\begin{eqnarray*}
Pr\{\overline{X_R}^2 + 2\frac{D}{C}\overline{X_R}\overline{Y_R} + \overline{Y_R}^2 > \frac{2(C^2-D^2}{C}\} &\leq & \frac{1}{2} \\
Pr\{\overline{X_R}^2 + 2\frac{D}{C}\overline{X_R}\overline{Y_R} + \overline{Y_R}^2 \leq \\frac{2(C^2-D^2}{C}\} &>& \frac{1}{2}.
\end{eqnarray*}

Chebyshev tells us that at least half of the distribution lies in the ellipse centered at the origin defined by the inequality above.  
By a similar computation as with the Gaussian distribution, the squares of the dimensions of this ellipse are $2(C+D)$ and $2(C-D)$, 
both of which are less than $\frac{4n}{3}$.  Hence, the ellipse lies inside a circle of radius $2\sqrt{\frac{n}{3}} < \sqrt{2n}$, and 
therefore over half of the distribution of the remainder must lie inside this circle.
\end{proof}

\subsection{The Proof}
We are ready to prove the main theorem.

\begin{proof}{Theorem~\ref{versatilePowers}}
Suppose that $\alpha$ has $m \geq c n$ blocks in its binary representation.  
Then either there are $0.01m$ blocks of length at least $2$ or there are fewer than $0.01m$ blocks of length at least $2$.  
In the second case, Lemma~\ref{manyBlocks} tells us that we can find $0.49m$ 
identical pairs of consecutive blocks of length $1$.  Lemma~\ref{manyType4} then says that with probability at least 
$\frac{1}{6}$, the carry bits arrange themselves in such a way that there are at least 
$\frac{0.49}{4}m \geq \frac{1}{9}m$ identical type $4$ distributions.  

If there are $m' > \frac{1}{100}m$ blocks of length at least $2$, then Lemma~\ref{manyType4} says that with 
probability at least $\frac{1}{6}$, the carry bits arrange themselves in such a way that there are at least 
$\frac{1}{4}m' > \frac{1}{400}m$ type $4$ distributions.  Since $\frac{1}{9}>\frac{1}{400}$, 
we conclude that for any $\alpha$ with $m$ blocks, 
we can find $\frac{1}{400}m$ type $4$ distributions with probability $\frac{1}{6}$.

As $m \geq c n$, the number of identical type $4$ distributions exceeds $\frac{1}{400}m \geq \frac{c}{400} n$.  
By Lemma~\ref{identicalDists}, we can find $l \geq \frac{c}{1600}\cdot\frac{1}{400}m \geq \frac{c^2}{640000}n$ 
identical type $4$ distributions each with block length $L$.  By Lemma~\ref{gaussDist}, 
these sum to a Gaussian whose probability distribution function $f_G(x,y)\geq \frac{1}{\pi n} e^{-\frac{144n}{l}}$ 
inside a circle of radius $4\sqrt{n}$ centered at the mean of $G$.  As each type $4$ distribution has mean 
$\left(\frac{L}{2},\frac{L}{2}\right) \pm \frac{L}{2}\cdot\frac{1}{2^L-1}(1,-1)$, we decompose the Gaussian into 
a Gaussian $G=(X_G,Y_G)$ centered at $\left(\frac{Ll}{2},\frac{Ll}{2}\right)$ and a translation in the $(1,-1)$ 
direction which contributes to the translation term $(X_T,Y_T)$. 

It is worth noting that every distribution type of block length $L$ can be decomposed into the sum of a 
distribution centered at $\left(\frac{L}{2},\frac{L}{2}\right)$ and a translation in the $(1,-1)$ direction.  
To see this, we will write the mean of each type of distribution as $\left(\frac{L}{2},\frac{L}{2}\right) + k(1,-1)$, 
for some $k$ depending on $L$.  

Type $1$ distributions have mean $\left(\frac{L}{2},\frac{L}{2}\right)$.  Type $2$ distributions have mean 
$\left(\frac{L}{2},\frac{L}{2}\right) \pm \frac{L}{2}(1,-1)$.  Type $3$ distributions have mean 
$\left(\frac{L}{2},\frac{L}{2}\right) \pm \frac{L}{2}(1,-1)$.  Type $4$ distributions have mean 
$\left(\frac{L}{2},\frac{L}{2}\right) \pm \frac{L}{2}\cdot\frac{1}{2^L-1}(1,-1)$.

We extract the translation component from each term and call that sum $(X_T,Y_T)$.  
Let $R = (X_R,Y_R) = (X,Y) - (X_G,Y_G) - (X_T,Y_T)$ be the remainder.  If $L_T$ 
denotes the total length of all blocks contributing to $R$, we have that $\left(\frac{L_T}{2},\frac{L_T}{2}\right)$ 
is the mean of $R$.  Let $\overline{R} = R-\left(\frac{L_T}{2},\frac{L_T}{2}\right)$.  
By Lemma~\ref{remainder}, at least half of the distribution of $\overline{R}$ lies in an circle centered at the origin 
with radius $\sqrt{2n}$.  By taking the square $W$ of side length $2\sqrt{2n}$ surrounding the circle, we see that at least half 
of the distribution of $R$ lies in $W$.  

For any point $(p,q)$ in the square, consider the distribution of the Gaussian $\overline{G}+(p,q) = (\overline{X_G}+p,\overline{Y_G}+q)$, 
which is centered at $(p,q)$.  Lemma~\ref{gaussDist} guarantees that the probability distribution function exceeds 
$\frac{1}{\pi n}e^{-\frac{144n}{l}}$ inside a circle of radius $4\sqrt{n}$ centered at $(p,q)$.  
Contained within this circle is a square with side length $\sqrt{2n}$ in the second quadrant.  
Hence, the probability that $\overline{G}+(p,q)$ lies in the second quadrant is at least 
$\frac{2}{\pi}e^{-\frac{144n}{l}}$.

Recall that $l \geq \frac{c^2}{640000}n$, where $m \geq c n$.  So we have: 

\begin{eqnarray*}
\frac{2}{\pi}e^{-\frac{144n}{l}} &\geq & \frac{2}{\pi}e^{-\frac{144\cdot 640000}{c^2}}\\
&= & \frac{2}{\pi}e^{-\frac{92160000}{c^2}}.
\end{eqnarray*}

Take $C = \frac{2}{\pi}e^{-\frac{92160000}{c^2}}$.  Then with probability $\frac{1}{6}$, at least 
$C$ fraction of the distribution of $\overline{G} + \overline{R}$ conditioned on the carry bits lies 
in the second quadrant, where $C$ is a constant depending only on $c$.  
By symmetry, the same fraction lies in the fourth quadrant.  Finally, we must add 
the translation $(X_T,Y_T)$ in the $(1,-1)$ direction.  No matter the size of the translation, we 
are guaranteed $C$ fraction in either the second or fourth quadrant.  Hence, we have at least 
$\frac{C}{6}$ of the unconditioned distribution of $(X,Y) = (X_G,Y_G) + (X_R,Y_R) + (X_T,Y_T)$ lying in the 
second or fourth quadrants relative to the mean $\left(\frac{n}{2},\frac{n}{2}\right)$, and so 
at least $\frac{C}{12}$ lying in either the second or fourth quadrant.  By symmetry, we get 
at least $\frac{C}{12}$ lying in both the second and fourth quadrants.
\end{proof}

\section{Heavily Shifting Numbers}
We have shown that $\alpha$ with many uniform blocks of $0$'s and $1$'s have the shifting 
property.  An interesting related question is whether there is an $\alpha$ that is heavily 
shifting: that is, $\alpha$ shifts almost all of the light strings to heavy strings.  More precisely, 
$o(1)$ fraction of light strings remain light under translation by $\alpha$.  We already 
know that when $\alpha$ has $o(\sqrt{n})$ blocks, $\alpha$ does not have 
the $\epsilon$-shifting property, and can therefore not be heavily shifting.  

Our current understanding of the joint initial and final weight distribution cannot quite show that $\alpha$ 
with $\Theta(n)$ blocks are also not heavily shifting.  
The reason is that we have no handle on the size of the translation term $(X_T,Y_T)$ in the $(1,-1)$ direction.  
It is possible that the translation is so large most of the time to make $\alpha$ heavily shifting, though we 
suspect this does not happen.  
  
It is an open problem to figure out which $\alpha$ are heavily shifting.

\section{Acknowledgements}
This material is based upon work supported by the National Science Foundation Graduate Research Fellowship 
under Grant No. DGE-1433187.  
The author would also like to thank Swastik Kopparty for many useful discussions.

\end{document}